\documentclass[11pt]{amsart}
\usepackage[margin=2.9cm]{geometry}
\usepackage{indentfirst}  
\usepackage{times}
\usepackage{mathrsfs}
\usepackage{latexsym}
\usepackage[dvips]{graphics}
\usepackage{enumitem}
\usepackage{epsfig}
\usepackage{hyperref, amsmath, amsthm, amsfonts, amscd, flafter,epsf}
\usepackage{amsmath,amsfonts,amsthm,amssymb,amscd}
\usepackage{color}
\usepackage{amssymb} 
\usepackage{mathtools}
\usepackage{tikz} 
\usepackage[utf8]{inputenc}

\usepackage{cleveref}
\crefname{section}{§}{§§}
\Crefname{section}{§}{§§}
\newcommand\be{\begin{equation}}
\newcommand\ee{\end{equation}}
\newcommand\bea{\begin{eqnarray}}
\newcommand\eea{\end{eqnarray}}
\newcommand\bi{\begin{itemize}}
	\newcommand\ei{\end{itemize}}
\newcommand\ben{\begin{enumerate}}
	\newcommand\een{\end{enumerate}}

\newtheorem{theorem}{Theorem}[section]

\newtheorem{corollary}[theorem]{Corollary}
\newtheorem{lemma}[theorem]{Lemma}
\newtheorem{proposition}[theorem]{Proposition}

\newtheorem{definition}[theorem]{Definition}

\newtheorem{remark}[theorem]{Remark}

\makeatletter
\renewenvironment{proof}[1][\proofname]{\par%
	\pushQED{\qed}%
	\normalfont \topsep6\p@\@plus6\p@\relax%
	\trivlist%
	\item[\hskip\labelsep%
	#1]\ignorespaces%
}{%
	\popQED\endtrivlist\@endpefalse%
}
\makeatother

\newcommand{\Z}{\ensuremath{\mathbb{Z}}}
\newcommand{\Q}{\mathbb{Q}}

\newcommand{\N}{\mathbb{N}}

\newcommand{\F}{\mathbb{F}}

\newcommand{\cO}{\mathcal{O}}

\newcommand{\Qp}{\mathbb{Q}_p}
\newcommand{\GL}{\text{GL}}
\newcommand{\Tor}{\text{Tor}}
\newcommand{\ind}{\text{ind}}

\newcommand{\mtwo}[4]{\begin{pmatrix}
	#1&#2\\#3&#4
\end{pmatrix}}
\newcommand{\Uni}[1]{\begin{pmatrix}
	1&#1\\&1
\end{pmatrix}}

\numberwithin{equation}{section}
\begin{document}
\title{A note on presentations of supersingular representations of $\GL_2(F)$}
\author{Zhixiang Wu}
\address{Laboratoire de Mathématiques d’Orsay, Unsiv.  Paris-Sud, Université Paris-Saclay, 91405 Orsay, France}
\email{zhixiang.wu@math.u-psud.fr}
%\subjclass[2010]{Primary: 05C??. Secondary: 05C??}
%\pagestyle{empty} 
%\setcounter{secnumdepth}{2}
%\setcounter{tocdepth}{1}
\maketitle
\begin{abstract}
We prove that any smooth irreducible supersingular representation with central character of $\GL_2(F)$ is never of finite presentation when $F$ is a finite field extension of $\Q_p$ such that $F\neq \Qp$, extending a result of Schraen in \cite{schraen2015presentation} for quadratic extensions.
\end{abstract}
\section{Introduction}
Let $p$ be a prime number. Let $F$ be a finite extension of $\Qp$ with ring of integers $\cO$. Let $n\geq 2$ be an integer. Recent years, several progresses have been made on the study of representations of $p$-adic Lie groups on vector spaces over fields of characteristic $p$,  motivated by the $p$-adic and mod-$p$ Langlands programs. The classifications of mod-$p$ irreducible admissible smooth representations of $\GL_n(F)$ in terms of supersingular representations was proved by Barthel-Livn\'e for $\GL_2$ (\cite{barthel1994}) and by Herzig for general $\GL_n$ (\cite{herzig2011gln}), which are now known for general reductive groups (\cite{abe2014reductive}). Supersingular representations of $\GL_2(\Qp)$ was classified by Breuil and some mod-$p$ Langlands correspondences appeared (\cite{Breuil2003gl2}). 
Up to now, except $\GL_2(\Qp)$ and some related groups such as $\text{SL}_2(\Qp)$ (\cite{abdellatif2011autour},\cite{cheng2013mod},\cite{koziol2016classification}), supersingular representations for general groups (e.g. $\GL_3(\Qp)$ or $\GL_2(F)$ when $F\neq \Qp$) remain mysterious. Some complexity of classifications of supersingular representations of $\GL_2(F)$ when $F\neq \Qp$ was shown by Breuil-Pa\v sk\=unas's construction of supersingular representations (\cite{breuil2012supersingular}). Daniel Le also constructed some non-admissible irreducible smooth mod-$p$ representations for certain $\GL_2(F)$ (\cite{le2018some}).\par
Let $G=\GL_2(F), K=\GL_2(\cO)$ and $Z$ be the center of $G$. Let $\pi$ be an irreducible smooth representation of $G$ over an algebraically closed characteristic $p$ field $k$ with central character. Then $\pi$ contains a smooth irreducible sub-representation $\sigma$ of subgroup $KZ$ and there is a surjective morphism of $G$-representations $\text{ind}_{KZ}^G\sigma\twoheadrightarrow \pi $ by the Frobenius reciprocity
where $\text{ind}_{KZ}^G\sigma$ denotes the compact induced representation. The representation $\pi$ is called of finite presentation if the kernel of the surjection $\text{ind}_{KZ}^G\sigma\twoheadrightarrow \pi $ is finitely generated as a $k[G]$-module. Such kind of finite presentation of representations of $G$ when $G=\GL_2(\Qp)$ are used by Colmez to construct a functor to get étale $(\varphi,\Gamma)$-modules from representations of $\text{GL}_2(\Qp)$, which plays a key role in mod-$p$ and $p$-adic Langlands correspondences for $\GL_2(\Qp)$ (\cite{colmez2010gl2}). Vign\'eras constructed a generalized functor from representations of $\GL_2(F)$ of finite presentation to étale $(\varphi,\Gamma)$-modules of finite type (\cite{Vigneras2011colmez}). Unfortunately, Schraen proved in \cite{schraen2015presentation} that any smooth irreducible supersingular representation with central character of $\GL_2(F)$ is never of finite presentation when $F$ is a quadratic field extension of $\Qp$. The proof relies on a kind of coherent rings found by Emerton (\cite{emerton2008class}) and a criterion of finite presentation for representations of $\GL_2$ by Hu (Theorem 1.3, \cite{hu2012diagrammes}). In the note, we extend the result for any finite field extension $F$ of $\Qp$ such that $F\neq \Qp$. 
\begin{theorem}[\ref{maintheorem}]
	If $[F:\Qp]\geq 2$, a smooth supersingular representation of $\GL_2(F)$ with central character is not of finite presentation. 
\end{theorem}
The proof firstly follows and simplifies the original arguments in \cite{schraen2015presentation}. Let $\ind_{KZ}^G\sigma/T(\ind_{KZ}^G\sigma)$ be the universal supersingular representation of $G$ where $T$ is the distinguished Hecke operator (cf. \cite{barthel1994}). Let $L(\sigma)$ be the subspace of $\ind_{KZ}^G\sigma/T(\ind_{KZ}^G\sigma)$ generated by $\sigma$ under the action of monoid $\mtwo{\varpi^{2\N}}{\cO}{}{1}$, where $\varpi$ is a uniformizer of $\cO$. Let $U:=\mtwo{1}{\cO}{}{1}$ be the subgroup of unipotent upper-triangular matrices in $\GL_2(\cO)$. Using some arguments on modules over coherent rings (Lemma \ref{lammatech1} and Lemma \ref{h0>0}), we prove that $\pi$ is not of finite presentation if the sub-module $L(\sigma)$ is not admissible, which means that the space $L(\sigma)^{U}$ of the $U$-invariants in $L(\sigma)$ is infinite-dimensional over $k$. The non-admissibility of $L(\sigma)$ is proved by explicitly finding invariant elements which is similar to works in \cite{Breuil2003gl2}, \cite{schein2011irreducibility}, \cite{morra2012some} and \cite{hendel2019universal}. A key observation is that the module structure of $L(\sigma)$ over the coherent ring guarantees that $\dim_k L(\sigma)^{U}=\infty$ if $\dim_k L(\sigma)^{U}\geq 2$. As a corollary, following \cite{emerton2008class} and \cite{schraen2015presentation}, our result gives a uniform proof for the following fact.
\begin{corollary}[\ref{maincor}]
	For any smooth irreducible representation $\sigma$ of $KZ$, the universal supersingular representation $\ind_{KZ}^G\sigma/T(\ind_{KZ}^G\sigma)$ of $\GL_2(F)$ is not admissible if $F\neq \Qp$.
\end{corollary}
\textit{Organization of the note}. In \cref{secrepresentationgl2}, we recall basic facts on mod-$p$ representations of $\GL_2(F)$ and Emerton's coherent rings. We prove the main result in \cref{secmain} with the proof for non-admissibility postponed to \cref{secnonadmissible}. 
\par
\textit{Notations}. We fix a uniformizer $\varpi$ of $F$. Let $k_F$ be the residue field of $\cO$. Let $d=[F:\Qp]$, $f=[k_F:\F_p]$, $e=d/f$ and $q=p^f$. Let $G=\GL_2(F)$,  $K=\GL_2(\cO)$ and $Z$ be the center of $G$. Let $K_1$ be the kernel of the reduction map $K\rightarrow \GL_2(k_F)$. Let $U=\left\{\mtwo{1}{a}{}{1},a\in \cO\right\}$ and $\alpha=\mtwo{\varpi}{}{}{1}$. Let $k$ be an algebraically closed field of characteristic $p$. We identify $k_F=\F_q$ and fix an embedding $k_F\hookrightarrow k$. All the representations in the note are on vector spaces over $k$.\par
\textit{Acknowledgement}. The author would like to express his sincere gratitude to his advisor Prof. Benjamin Schraen for suggesting the problem and for helpful discussions. The author would like to thank the anonymous referees for their comments and suggestions. The author thanks the Fondation Math\'ematique Jacques Hadamard (FMJH) and University of Paris-Sud for support.
\section{Preliminary on representations and coherent rings}\label{secrepresentationgl2}
\subsection*{Mod-$p$ representations of $\GL_2$}
We recall some results and notations in \cite{barthel1994} and \cite{Breuil2003gl2}.
Let $\pi$ be a smooth irreducible representation of $G$ with central character over $k$. Then $\pi$ contains an irreducible sub-$KZ$-representation $\sigma$ of $KZ$. Let $\ind_{KZ}^G\sigma$ be the compactly induced representation: the representation space consists of functions $f:G\rightarrow \sigma$ such that $f$ is compactly supported modulo $KZ$ and $f(k\cdot)=k.f(\cdot)$ for any $k\in KZ$ and the action of $G$ is given by right translations. There is a distinguished element $T\in\text{End}_G(\ind_{KZ}^G\sigma)$ which generates the Hecke algebra. By the definition and the classification in \cite{barthel1994}, $\pi$ is supersingular if and only if there exists a surjection $\ind_{KZ}^G\sigma\twoheadrightarrow \pi$ induced by an inclusion $\sigma\hookrightarrow \pi|_{KZ}$ and the Frobenius reciprocity such that the surjection factors through a map 
$$\ind_{KZ}^G\sigma/T(\ind_{KZ}^G\sigma) \twoheadrightarrow \pi$$
for some or every such $\sigma$.
\par 
If $0\leq r\leq p-1$ is an integer, let $\text{Sym}^r$ be the $r$-th symmetric power of the standard representation of $\GL_2(\F_q)$ on two-dimensional space $k^{2}$ via the embedding $\F_q\hookrightarrow k$. If $\vec{r}=(r_0,\cdots,r_{f-1})\in \Z^f$ with $0\leq r_j\leq p-1$ for any $0\leq j\leq f-1$,
we get a representation $\text{Sym}^{\vec{r}}:=\otimes_{j=0}^{f-1}\text{Sym}^{r_j}\circ \text{Fr}^j$, where $\text{Fr}$ denotes the automorphism of $\GL_2(\F_q)$ induced by the Frobenius automorphism of $\F_q$. If $\vec{a},\vec{b}\in \Z^f$, we say $\vec{a}\leq \vec{b}$ if $a_j\leq b_j$ for any $j=0,\cdots, f-1$. The representation $\text{Sym}^{\vec{r}}$ has a model consisting of homogeneous polynomials spanned by a basis $\{\otimes_{j=0}^{f-1}x_j^{r_j-i_j}y_j^{i_j}\}_{0\leq \vec{i}\leq \vec{r}}$. The group action is given by $$\mtwo{a}{b}{c}{d}.\otimes_{j=0}^{f-1}x_j^{r_j-i_j}y_j^{i_j}=\otimes_{j=0}^{f-1}(a^{p^j}x_j+c^{p^j}y_j)^{r_j-i_j}(b^{p^j}x_j+d^{p^j}y_j)^{i_j},$$
for any $0 \leq \vec{i}\leq \vec{r},  \mtwo{a}{b}{c}{d}\in \GL_2(\F_q).$ We abbreviate $x^{\vec{r}-\vec{i}}y^{\vec{i}}:=\otimes_{j=0}^{f-1}x_j^{r_j-i_j}y_j^{i_j}$. If $\chi:\F_q^{\times}\rightarrow k^{\times}$ is a character of $\F_q^{\times}$, $\chi\circ \text{det}$ is a character of $\GL_2(\F_q)$. 
We can naturally inflate the representation $(\chi\circ \text{det})\otimes\text{Sym}^{\vec{r}}$ of $\GL_2(\F_q)$ to a representation of $K$ by letting $K_1$ act trivially. Then the smooth irreducible $KZ$-representation $\sigma$ is isomorphic to $(\chi\circ \text{det})\otimes\text{Sym}^{\vec{r}}$ when restricted to $K$ for a unique $\chi:\F_q^{\times}\rightarrow k^{\times}$ and $\vec{r}$ as above and the action of $\mtwo{\varpi}{}{}{\varpi}\in Z$ on $\sigma$ is given by a scalar $\nu \in k^{\times}$.\par
If $g\in G, w\in \sigma$, let $[g,w]\in \ind_{KZ}^G\sigma$ be the element given by  
$$[g,w](g')=\left\{\begin{aligned}
	g'g.w & &\text{if } g'g\in KZ,\\
	0 & &\text{if  } g'g\notin KZ.
\end{aligned}
\right.	$$
Then $g'.[g,w]=[g'g,w], \forall g',g\in G, w\in\sigma$. If $S\subset G$ is a subset, let $[S,\sigma]$ be the subspace of $\ind_{KZ}^G\sigma$ spanned by $[g,w],w\in \sigma, g\in S$. \par
If $\lambda\in \mathbb{F}_q$, we let $[\lambda]$ be the Teichmüller lift of $\lambda$ in $F$. For any integer $n\geq 1$, the set  $I_n:=\{[\lambda_0]+\varpi[\lambda_1]+\cdots +\varpi^{n-1}[\lambda_{n-1}],\lambda_i\in \mathbb{F}_q\} $ is a complete set of representatives of $\cO/\varpi^n\cO$. We define $I_0=\{0\}$. If $\lambda=[\lambda_0]+\varpi[\lambda_1]+\cdots +\varpi^{n-1}[\lambda_{n-1}]\in I_n$, let  $[\lambda]_{n-1}:=\lambda-\varpi^{n-1}[\lambda_{n-1}]\in I_{n-1}$. If $\vec{i}\in\Z^f,\lambda\in \F_q$, we use the notation $\lambda^{\vec{i}}:=\lambda^{\sum_{0\leq j\leq f-1}p^ji_j}$. The element $\left[\mtwo{\varpi^n}{\lambda}{}{1},\sum_{\vec{0} \leq \vec{i}\leq\vec{r}}u_{\vec{i}}x^{\vec{r}-\vec{i}}y^{\vec{i}}\right]\in \ind_{KZ}^G\sigma$ where $u_{\vec{i}}\in k$ for any $\vec{i}, n\in \N,\lambda\in I_n$. The action of the operator $T$ on the element is calculated as in \cite{Breuil2003gl2} (or see Proposition 2.1, \cite{hendel2019universal}). If $n\geq 1, \mu\in I_n$,
\begin{equation}\label{equationofT}
\begin{aligned}
	T\left(\left[\mtwo{\varpi^n}{\mu}{}{1}, \sum_{0\leq\vec{i}\leq \vec{r}} u_{\vec{i}}x^{\vec{r}-\vec{i}}y^{\vec{i}}\right]\right)=& \sum_{\lambda \in \F_q}\left[\mtwo{\varpi^{n+1}}{\mu+\varpi^n[\lambda]}{}{1}, (\sum_{0\leq\vec{i}\leq \vec{r}}u_{\vec{i}}(-\lambda)^{\vec{i}})x^{\vec{r}}\right]\\
&+\nu\left[\mtwo{\varpi^{n-1}}{[\mu]_{n-1}}{}{1}, u_{\vec{r}}\otimes_{j=0}^{f-1} (\mu_{n-1}^{p^j}x_j+y_j)^{r_j}\right],\\
T\left(\left[\mtwo{1}{}{}{1}, \sum_{0\leq\vec{i}\leq \vec{r}} u_{\vec{i}}x^{\vec{r}-\vec{i}}y^{\vec{i}}\right]\right)=& \sum_{\lambda \in \F_q}\left[\mtwo{\varpi}{[\lambda]}{}{1}, (\sum_{0\leq\vec{i}\leq \vec{r}}u_{\vec{i}}(-\lambda)^{\vec{i}})x^{\vec{r}}\right]\\
&+\left[\mtwo{1}{}{}{\varpi}, u_{\vec{r}}y^{\vec{r}}\right].
\end{aligned}
\end{equation}
\subsection*{A class of coherent rings}
We now recall some results in \cite{emerton2008class} and \cite{schraen2015presentation} on a type of coherent rings and their applications on representations of $\GL_2$. Assume $A$ is a complete regular local ring of dimension $d$ with residue field $k$ and maximal ideal $\mathfrak{m}$. Assume $\phi:A\rightarrow A$ is a local flat ring endomorphism of $A$ and assume $\phi$ is equal to the identity map on $k$ after reduction modulo $\mathfrak{m}$. We let $A[X]_{\phi}$ be the ring of polynomials in variable $X$ with commutative relation $Xa=\phi(a)X, \forall a\in A$. By Proposition 1.3 in \cite{emerton2008class}, $A[X]_{\phi}$ is a coherent ring which means that any finitely generated submodule of a finitely presented left $A[X]_{\phi}$-module is finitely presented.\par 
Modulo $\mathfrak{m}$, we get a ring morphism $A[X]_{\phi}\rightarrow k[X]$. If $M$ is a left  $A[X]_{\phi}$-module, there are natural isomorphisms $\Tor^{A[X]_{\phi}}_{i}(k[X],M)\simeq \Tor^A_{i}(k,M)$ for all $i\geq 0$ (Lemma 2.1, \cite{emerton2008class}). The isomorphisms equip the $k$-spaces $\Tor^A_{i}(k,M)$ $k[X]$-module structures. If $M$ is a finitely presented $A[X]_{\phi}$-module, then for any $i\geq 0$, $\Tor_i^A(k,M)$ is a finitely generated $k[X]$-module (Proposition 2.2, \cite{emerton2008class}).\par
An $A$-module is called smooth if any finitely generated submodule is Artinian. An $A[X]_{\phi}$-module is called smooth if the underlying $A$-module is smooth. 
If $M$ is an $A$-module, we let $M[\mathfrak{m}]=\{x\in M\mid mx=0,\forall m\in\mathfrak{m}\}$. There is a non-canonical isomorphism between functor $M\mapsto M[\mathfrak{m}]$ and functor $M\mapsto \Tor_d^A(k,M)$. An $A$-module $M$ is called admissible if it is smooth and $M[\mathfrak{m}]\simeq\Tor_d^A(k,M)$ is finite-dimensional over $k$. An $A[X]_{\phi}$-module is called admissible if the underlying $A$-module is admissible. \par

From now on, we let $A:=k[[U]]=\varprojlim k[U/N]$, where $N$ ranges over all open normal subgroups of $U$, be the Iwasawa algebra of $U$. Then $A\simeq k[[X_1,\cdots,X_d]]$ the ring of formal power series in $d$ variables with maximal ideal $\mathfrak{m}=(X_1,\cdots,X_d)$. The action of $\alpha=\mtwo{\varpi}{}{}{1}$ on $U:u\mapsto \alpha u\alpha^{-1}$ induces a flat local morphism $\phi:A\rightarrow A$. If $\Pi$ is a smooth representation of $U$, $\Pi$ is naturally a smooth $A$-module and $\Pi^{U}=\Pi[\mathfrak{m}]$. Thus the representation $\Pi$ is an admissible $A$-module if and only if $\Pi$ is an admissible $U$-representation. Any representation $\Pi$ of monoid $\mtwo{\varpi^{\N}}{\cO}{}{1}$ is now naturally an $A[X]_{\phi}$-module where $X$ acts by the action of $\alpha$ on $\Pi$.\par
Let $\sigma$ be an irreducible smooth representation of $KZ$. For any $n\geq 0$, let $R_{n}(\sigma):=\left[\mtwo{\varpi^{n}}{\cO}{}{1},\sigma\right]$ which is a sub-$A$-module of $\ind_{KZ}^G\sigma$. For any $k\in\N$, we let 
\begin{align*} I_{\geq k}(\sigma)&:=\bigoplus_{n \geq k}R_n({\sigma}),\quad
	I_{\geq k}^e(\sigma):=\bigoplus_{n \geq k,2\mid n}R_n({\sigma}),\quad
	I_{\geq k}^o(\sigma):=\bigoplus_{n \geq k,2\nmid n}R_n({\sigma}),
\end{align*}
be subspaces of $\ind_{KZ}^G\sigma$. We let $\phi_2:=\phi^2:A\rightarrow A$.
We have (Lemma 2.11, \cite{schraen2015presentation})
\begin{align*}
I_{\geq 0}^e(\sigma)\simeq A[X]_{\phi_2}\otimes_{A}\sigma,\quad
I_{\geq 1}^o(\sigma)\simeq A[X]_{\phi_2}\otimes_{A}R_1(\sigma)
\end{align*}
as $A[X]_{\phi_2}$-modules. \par
By the formula of the operator $T$ (\ref{equationofT}), we have $T(R_{n}(\sigma))\subset R_{n+1}(\sigma)\oplus R_{n-1}(\sigma)$ if $n\geq 1$. Hence $T(I_{\geq 1}(\sigma))\subset I_{\geq 0}(\sigma)$ and $T(I_{\geq 1}^o(\sigma))\subset I_{\geq 0}^e(\sigma)$, etc. We decompose $T|_{I_{\geq 1}(\sigma)}=T_++ T_-$ by the decomposition $T|_{R_n(\sigma)}=T_+|_{R_n(\sigma)}+ T_-|_{R_n(\sigma)}$, where $T_+|_{R_n(\sigma)}:R_n(\sigma)\rightarrow R_{n+1}(\sigma)$ and $T_-|_{R_n(\sigma)}:R_n(\sigma)\rightarrow R_{n-1}(\sigma)$ are compositions of the projections to the direct sum factors of $R_{n+1}(\sigma)\oplus R_{n-1}(\sigma)$ and $T|_{R_n(\sigma)}$, for all $n\geq 1$. \par
Let $L(\sigma):=I_{\geq 0}^e(\sigma)/T(I_{\geq 1}^o(\sigma))$. Then $L(\sigma)$ is an $A[X]_{\phi_2}$-module. The following proposition is essentially Proposition 2.23 in \cite{schraen2015presentation} which we recall the proof.
\begin{proposition}\label{proptord0L}
	$\Tor_0^A(k,L(\sigma))=0$. The $k[X]$-torsion part of $\Tor_d^A(k,L(\sigma))$ is isomorphic to $k=k[X]/(X)$ and coincides with the image of $\Tor_d^A(k,\sigma)$ via the morphism $\sigma\hookrightarrow I_{\geq 0}^e(\sigma)\twoheadrightarrow L(\sigma)$.
\end{proposition}
\begin{proof}
	We have an exact sequence
	\begin{align*}
		0\rightarrow \Tor_d^A(k,I_{\geq 1}^o(\sigma))\stackrel{\Tor_d^A(T)}{\rightarrow} \Tor_d^A(k,I_{\geq 0}^e(\sigma))\rightarrow \Tor_d^A(k,L(\sigma))\rightarrow \Tor_{d-1}^A(k,I_{\geq 1}^o(\sigma)) \cdots\\
			\cdots\rightarrow \Tor_0^A(k,I_{\geq 1}^o(\sigma))\stackrel{\Tor_0^A(T)}{\rightarrow} \Tor_0^A(k,I_{\geq 0}^e(\sigma))\rightarrow \Tor_0^A(k,L(\sigma))\rightarrow 0.
	\end{align*}
	And $\Tor_i^A(k,I_{\geq 0}^e(\sigma))\simeq \oplus_{k\geq 0}\Tor_i^A(k,R_{2k}(\sigma)),\Tor_i^A(k,I_{\geq 1}^o(\sigma))\simeq \oplus_{k\geq 0}\Tor_i^A(k,R_{2k+1}(\sigma))$ for $i\in\N$. \par
	By Lemma 2.12 in \cite{schraen2015presentation}, $\Tor_0^A(T_+)=0$, $\Tor_0^A(T_-)=\Tor_0^A(T)$ and $\Tor_0^A(T_-)$ sends each $\Tor_0^A(k,R_{2k+1}(\sigma))$ onto $\Tor_0^A(k,R_{2k}(\sigma))$. Hence $\Tor_0^A(T)$ in the above diagram is a surjection and $\Tor_0^A(k,L(\sigma))=0$. Since $\Tor_{d-1}^A(k,I_{\geq 1}^o(\sigma))\simeq \Tor_{d-1}^A(k,A[X]_{\phi_2}\otimes_A (A \otimes_{\phi,A}\sigma))\simeq k[X]\otimes_k \Tor_{d-1}^A(k,A \otimes_{\phi,A}\sigma)$ by Proposition 1.4 in \cite{schraen2015presentation}, the $k[X]$-module $\Tor_{d-1}^A(k,I_{\geq 1}^o(\sigma))$ is torsion free. Hence $\Tor_d^A(k,L(\sigma))_{tors}=\text{coker}(\Tor_d^A(T))_{tors}$. By Lemma 2.12 in \cite{schraen2015presentation} again, $\Tor_d^A(T_-)=0$ and $\Tor_d^A(T_+)$ sending $\Tor_d^A(k,R_{2k+1}(\sigma))$ to $\Tor_d^A(k,R_{2k+2}(\sigma))$ is an isomorphism. Thus the image of $\Tor_d^A(T)$ in $\Tor_d^A(k,I_{\geq 0}^e(\sigma))$ is $\oplus_{k\geq 1}\Tor_d^A(k,R_{2k}(\sigma))$. Since $R_0(\sigma)=\sigma$, $\Tor_d^A(k,L(\sigma))_{tors}$ coincides with the image of $\Tor_d^A(k,\sigma)$ via the map $\sigma\hookrightarrow I_{\geq 0}^e(\sigma)\twoheadrightarrow L(\sigma)$ in $L(\sigma)$. Finally, $\Tor_d^A(k,\sigma)\simeq\sigma^U$ is one-dimensional over $k$ by Lemma 2 in \cite{barthel1994}.
\end{proof}
We recall the following key lemma on smooth finitely presented $A[X]_{\phi}$-modules in \cite{schraen2015presentation}.
\begin{lemma}[\cite{schraen2015presentation}, Lemma 1.13]\label{lemma1.13}
	Let $M$ be a smooth finitely presented $A[X]_{\phi}$-module. Then there exists an increasing sequence of sub-$A[X]_{\phi}$-modules $(M_i)_{i\geq 0}$, a sequence of finite-dimensional $k$-vector spaces $(V_i)_{i\geq 0}$ such that there exist isomorphisms $M_{i+1}/M_i\simeq A[X]_{\phi}\otimes_A V_i$ as $A[X]_{\phi}$-modules, and if we let $\widetilde{M}=\cup_i M_i$, then $\Tor_d^A(k,M)_{tors}\simeq \Tor_d^A(k,M/\widetilde{M})$. In particular, $M/\widetilde{M}$ is admissible and each $M_i$ is of finite presentation. 
\end{lemma}
\section{Presentations of supersingular representations}\label{secmain}
We prove some lemmas on $A[X]_{\phi}$-modules. 
\begin{lemma}\label{lammatech1}
Let $M$ be a non-zero, smooth, finitely presented $A[X]_{\phi}$-module. Assume that $\Tor^A_d(k,M)$ is a torsion free $k[X]$-module. Then $\Tor_0^A(k,M)$ is infinite-dimensional over $k$. 
\end{lemma}
\begin{proof}
	By Lemma \ref{lemma1.13}, we can find an increasing sequence of sub-$A[X]_{\phi}$-modules $(M_i)_{i\geq 0}$, a sequence of finite-dimensional $k$-vector spaces $(V_i)_{i\geq 0}$ such that there exist isomorphisms $M_{i+1}/M_i\simeq A[X]_{\phi}\otimes_A V_i$ of $A[X]_{\phi}$-modules with $M_0=0$, and if we let $\widetilde{M}=\cup_i M_i$, then $\Tor_d^A(k,M)_{tors}\simeq \Tor_d^A(k,M/\widetilde{M})$. Thus $\Tor_d^A(k,M/\widetilde{M})=0$ by assumptions. Hence $M=\widetilde{M}$ by Lemma 1.8 in \cite{schraen2015presentation}. Since $M$ is finitely generated, there exists a minimal $n\in \N$ such that $M=M_{n}$. Since $M$ is non-zero, we have $n\geq 1$ and $M_n\neq M_{n-1}$. We have a surjection $$M\twoheadrightarrow M/M_{n-1}\simeq A[X]_{\phi}\otimes_A V_{n-1}.$$ Thus we have a surjection 
	$$\Tor_0^A(k,M)\twoheadrightarrow \Tor_0^A(k,A[X]_{\phi}\otimes_A V_{n-1}).$$
	But by Proposition 1.4 and Example 1.6 in \cite{schraen2015presentation}, $\Tor_0^A(k,A[X]_{\phi}\otimes_A V_{n-1})\simeq k[X]\otimes\Tor_0^A(k,V_{n-1})$ is a free $k[X]$-module of rank $\text{dim}_k{V_{n-1}}$. Assume that $\Tor_0^A(k,M)$ is finite-dimensional over $k$. Then $V_{n-1}$ is zero by the surjection above. This contradicts that $M_n\neq M_{n-1}$. Hence $\Tor_0^A(k,M)$ is infinite-dimensional over $k$.
\end{proof} 

\begin{lemma}\label{h0>0}
Let $M$ be a smooth, finitely presented $A[X]_{\phi}$-module and $N$ be a non-zero sub-$A[X]_{\phi}$-module of $M$. Assume that $M/N$ is finitely presented and admissible, and $\Tor_d^A(k,N)$ is torsion free. Then $\Tor_0^A(k,M)$ is infinite-dimensional over $k$.
\end{lemma}
\begin{proof}
	Since $M/N$ and $M$ are finitely presented, by the coherence of  $A[X]_{\phi}$ (Proposition 1.3, \cite{emerton2008class}), $N$ is of finite presentation. Thus by Lemma \ref{lammatech1}, $\Tor_0^A(k,N)$ is infinite-dimensional over $k$. Consider the long exact sequence
	$$\cdots \rightarrow \Tor_1^A(k,M/N)\rightarrow \Tor_0^A(k,N)\rightarrow \Tor_0^A(k,M)\rightarrow \Tor_0^A(k,M/N)\rightarrow 0.$$
	Since $M/N$ is admissible, by Corollary 1.12 in \cite{schraen2015presentation}, $\Tor_1^A(k,M/N)$ and $\Tor_0^A(k,M/N)$ are finite-dimensional over $k$. Since $\Tor_0^A(k,N)$ is infinite-dimensional over $k$, so is $\Tor_0^A(k,M)$.
\end{proof}
\begin{definition}
	A smooth representation $\pi$ of $G$ is called of finite presentation if there exists an irreducible smooth representation $\sigma$ of $KZ$ and a surjection $$\ind_{KZ}^G\sigma\twoheadrightarrow \pi$$
	such that the kernel is finitely generated as a $k[G]$-module.
\end{definition}

\begin{remark}
	By Proposition 4.4 in \cite{hu2012diagrammes}, if $\pi$ is of finite presentation, then for all smooth finite-dimensional sub-$KZ$-representation $\sigma$ of $\pi$ which generates the $G$-representation $\pi$, the kernel of the surjection $\ind_{KZ}^G\sigma\twoheadrightarrow \pi$ is finitely generated as a $k[G]$-module.
\end{remark}
\begin{remark}
	If $F=\Qp$, then by the classifications in \cite{barthel1994} and \cite{Breuil2003gl2}, any irreducible representation of $\GL_2(\Qp)$ with central character is of finite presentation.
\end{remark}
Assume $\pi$ is a smooth irreducible representation of $G$ with central character, and $\sigma\subset \pi$ is an irreducible smooth sub-$KZ$-representation. Let $I^+(\pi,\sigma):=\mtwo{\varpi^{\N}}{\cO}{}{1}\sigma\subset \pi$ be the $A[X]_{\phi}$-submodule of $\pi$ generated by $\sigma$. Then $I^+(\pi,\sigma)$ is the image of $I_{\geq 0}(\sigma)$ in $\pi$ via the map $\ind_{KZ}^G\sigma\twoheadrightarrow \pi$. We recall the following result of Yongquan Hu. 
\begin{theorem}[\cite{hu2012diagrammes}, Theorem 1.3]\label{thmhu}
If $\pi$ is of finite presentation, then $I^+(\pi,\sigma)^U$ is a finite-dimensional $k$-vector space.
\end{theorem}
We will prove the following theorem in \cref{secnonadmissible}.
\begin{theorem}\label{theoremnonadmissiblility}
	The $A$-module $L(\sigma)$ is not admissible if $[F:\Qp]\geq 2$. In particular, the $k[X]$-module $\Tor_d^A(k,L(\sigma))$ is not torsion.
\end{theorem}
Now assuming Theorem \ref{theoremnonadmissiblility}, we prove the main theorem.
\begin{theorem}\label{maintheorem}
	If $\pi$ is a smooth supersingular representation of $\GL_2(F)$ with central character, then $\pi$ is not of finite presentation when $[F:\Qp]\geq 2$. 
\end{theorem}
\begin{proof}
	We can find a surjection $\ind_{KZ}^G(\sigma)/(T)\twoheadrightarrow{\pi}$ for some irreducible smooth sub-$KZ$-representation $\sigma$ of $\pi$ by the definition of supersingular representations. Let $I^+(\pi,\sigma)$ be the $A[X]_{\phi}$-submodule of $\pi$ generated by $\sigma$ and let $M(\pi,\sigma)$ be the $A[X]_{\phi_2}$-submodule of $\pi$ generated by $\sigma$. Then $M(\pi,\sigma)\subset I^+(\pi,\sigma)$. The map of $A[X]_{\phi_2}$-modules $I_{\geq 0}^e(\sigma)\hookrightarrow \ind_{KZ}^G\sigma\rightarrow \pi$ factors through $L(\sigma)\rightarrow \pi$ with image $M(\pi,\sigma)$. Let $N(\pi,\sigma)$ be the kernel of the morphism $L(\sigma)\rightarrow M(\pi,\sigma)$ of $A[X]_{\phi_2}$-modules. We have an exact sequence $$0\rightarrow \Tor_d^A(k,N(\pi,\sigma))\rightarrow \Tor_d^A(k,L(\sigma))\rightarrow \Tor_d^A(k,M(\pi,\sigma)).$$
	By Proposition \ref{proptord0L}, $\Tor_d^A(k,L(\sigma))_{tors}$ is generated by the image of $\sigma^U\simeq \Tor_d^A(k,\sigma)$ via the map $\sigma\rightarrow I_{\geq 0}^e(\sigma)\twoheadrightarrow L(\sigma)$. The non-zero composition map $\sigma \rightarrow L(\sigma)\rightarrow M(\pi, \sigma)$ induces morphisms $\Tor_d^A(k,\sigma)\stackrel{\sim}{\rightarrow} \Tor_d^A(k,L(\sigma))_{tors}\rightarrow \Tor_d^A(k,M(\pi,\sigma))$. The composition $\sigma\rightarrow M(\pi,\sigma)$ is injective since $\sigma$ is irreducible. Since $\Tor_d^A(k,-)$ is left exact, we get an injection $\Tor_d^A(k,L(\sigma))_{tors}\hookrightarrow \Tor_d^A(k,M(\pi,\sigma))$. Then $\Tor_d^A(k,N(\pi,\sigma))$ must be a torsion free $k[X]$-module. \par
	Now if $\pi$ is finitely presented, $M(\pi,\sigma )\subset I^+(\pi, \sigma)$ is admissible by Hu's result (Theorem \ref{thmhu}). Since $M(\pi,\sigma)$ is generated by $\sigma$, it is a finitely generated $A[X]_{\phi_2}$-module. Moreover, the proof of Theorem 2.24 in \cite{schraen2015presentation} shows that $M(\pi,\sigma)$ is of finite presentation  ($M(\pi,\sigma)$ is stable under the action of $H=\mtwo{\cO^{\times}}{}{}{1}$, then use Lemma 2.6 in \cite{schraen2015presentation}). If $N(\pi,\sigma)\neq 0$, then all the assumptions in Lemma \ref{h0>0} are satisfied if we take $M=L(\sigma)$ and $N=N(\pi,\sigma)$. Remark that here $L(\sigma),N(\pi,\sigma)$ are modules over the coherent ring $A[X]_{\phi_2}$ rather than $A[X]_{\phi}$, but Lemma \ref{h0>0} holds true for $A[X]_{\phi_2}$. Thus by Lemma \ref{h0>0}, $\Tor_0^A(k,L(\sigma))$ has infinite dimension over $k$, which contradicts that $\Tor_0^A(k,L(\sigma))=0$ (Proposition \ref{proptord0L})! Hence $N(\pi,\sigma)=0$. Then $L(\sigma)\simeq M(\pi,\sigma)$ is admissible. This contradicts Theorem \ref{theoremnonadmissiblility}! Hence $\pi$ is not of finite presentation. 
\end{proof}
\section{Non-admissibility}\label{secnonadmissible}
Assume $\sigma =\text{Sym}^{\vec{r}}\otimes (\chi\circ\text{det})$, where $\vec{r}=(r_0,\cdots,r_{f-1})$ such that $ 0 \leq r_0,\cdots,r_{f-1}\leq p-1$, is an irreducible representation of $KZ$ with $\varpi\in Z$ acting on $\sigma$ as a scalar $\nu\in k^{\times}$. Recall that $$R_1(\sigma)=\bigoplus_{\mu\in\F_q}\left[\mtwo{\varpi}{[\mu]}{}{1},\sigma\right],\quad R_2(\sigma)=\bigoplus_{\mu,\lambda\in\F_q}\left[\mtwo{\varpi^2}{[\mu]+\varpi[\lambda]}{}{1},\sigma\right].$$ For any $\mu\in\F_q, u_{\vec{i}}\in k, \vec{i}\in\Z^f, \vec{0} \leq\vec{i}\leq \vec{r}$, the operators $T_{\pm}$ act on $\left[\mtwo{\varpi}{[\mu]}{}{1}, \sum_{0\leq\vec{i}\leq \vec{r}} u_{\vec{i}}x^{\vec{r}-\vec{i}}y^{\vec{i}}\right]\in R_1(\sigma)$ by the formulas (see (\ref{equationofT})): \begin{align}\label{formulaT1}
	T_+\left(\left[\mtwo{\varpi}{[\mu]}{}{1}, \sum_{0\leq\vec{i}\leq \vec{r}} u_{\vec{i}}x^{\vec{r}-\vec{i}}y^{\vec{i}}\right]\right)&= \sum_{\lambda \in \F_q}\left[\mtwo{\varpi^2}{[\mu]+\varpi [\lambda]}{}{1}, (\sum_{0\leq\vec{i}\leq \vec{r}}u_{\vec{i}}(-\lambda)^{\vec{i}})x^{\vec{r}}\right]\\ \label{formulaT2}
	T_-\left(\left[\mtwo{\varpi}{[\mu]}{}{1}, \sum_{0\leq\vec{i}\leq \vec{r}} u_{\vec{i}}x^{\vec{r}-\vec{i}}y^{\vec{i}}\right]\right)&=\nu \left[\mtwo{1}{}{}{1}, u_{\vec{r}}\otimes_{j=0}^{f-1} (\mu^{p^j}x_j+y_j)^{r_j}\right].
\end{align}
%In fact, we only need to show that $\text{dim}_k((R_0\oplus R_2)/TR_1)^U$ has dimension at least $2$. 
\begin{proof}[Proof of Theorem \ref{theoremnonadmissiblility}]
	We need to prove that $L(\sigma)^U$ is infinite-dimensional over $k$. By Proposition \ref{proptord0L}, the torsion part of the $k[X]$-module $\Tor_d^A(k, L(\sigma))\simeq L(\sigma)^U$ has only dimension $1$. If $\text{dim}_kL(\sigma)^U\geq 2$, the free part of the $k[X]$-module $\Tor_d^A(k, L(\sigma))$ can not be zero and then $\Tor_d^A(k, L(\sigma))$ is infinite-dimensional over $k$ since a non-zero free $k[X]$-module is infinite-dimensional over $k$. So we only need to prove that $\text{dim}_kL(\sigma)^U\geq 2$ to show that $L(\sigma)$ is not an admissible $A$-module. 
	We will prove
	\begin{lemma}\label{mainlem}
		If $[F:\Q_p]\geq 2$, there exists an element $g\in R_2(\sigma)$ such that $g\notin T_+ R_1(\sigma)$ and $ug-g\in TR_1'(\sigma)$ for any $u\in U$, where $R_1'(\sigma)$ is the kernel of $T_-|_{R_1(\sigma)}$.
	\end{lemma}
	Now assume there exists an element $g$ as in Lemma \ref{mainlem}. Then the image of $g$ in $L(\sigma)$ lies in $L(\sigma)^U$ since $ug-g\in TR_1'(\sigma)\subset TI^o_{\geq 1}(\sigma)$ which is zero in $L(\sigma)=I^e_{\geq 0}(\sigma)/TI^o_{\geq 1}(\sigma)$ for any $u\in U$. We claim that the image of $g$ doesn't lie in the image of $R_0(\sigma)$ in $L(\sigma)$. Otherwise there exist $a\in R_0(\sigma), x\in I^{o}_{\geq 1}(\sigma)$ such that $g-a=Tx$. Assume $x=\sum_{k\in\N}x_{2k+1}$, where each $x_{2k+1}\in R_{2k+1}(\sigma)$ and there are only finitely many $k$ such that $x_{2k+1}\neq 0$. Since $g\notin T_+R_1(\sigma)$, $g\neq 0$ and we may assume $x\neq 0$. Let $k_0$ be the maximal integer such that $x_{2k_0+1}\neq 0$. Then $Tx=T_-(x_1)+\sum_{k=0}^{k_0}(T_+(x_{2k+1})+T_-(x_{2k+3}))\in R_0(\sigma)\oplus(\oplus_{k=0}^{k_0} R_{2k+2}(\sigma))$. Since $Tx=g-a\in R_0(\sigma)\oplus R_2(\sigma)$, if $k_0\neq 0$, $T_+(x_{2k_0+1})=0\in R_{2k_0+2}(\sigma)$. This contradicts that $T_+$ is injective (Lemma 2.12 in \cite{schraen2015presentation}) and $x_{2k_0+1}$ is not $0$. If $k_0=0$, then $g=T_+(x_1)\in T_+ R_1(\sigma)$, which contradicts our choice of $g$ in the Lemma \ref{mainlem}. Hence the image of $g$ in $L(\sigma)$ doesn't lie in the image of $R_0(\sigma)$ in $L(\sigma)$. Thus the image of $\sigma^U$ and $g$ in $L(\sigma)$ span a two-dimensional subspace of $L(\sigma)^U$. This proves that $\dim_k(L(\sigma))\geq 2$ and $L(\sigma)$ is not admissible.
\end{proof}
Before the proof of Lemma \ref{mainlem}, we remark the following simple facts.
\begin{lemma}\label{lemmasumfq}
	Let $F=\sum_{i}a_iX^i\in k[X]$ be a polynomial of degree no more than $q-1$, then $\sum_{t\in\F_q}F(t)=-a_{q-1}$. 
\end{lemma}
\begin{lemma}[\cite{schein2011irreducibility}, Lemma 2.2]\label{lemmawitt}
	For any $a,b\in\F_q$,
	$[a]+[b]\equiv [a+b] +\varpi^e[P(a,b)] \mod \varpi^{e+1},$
	where $P(a,b)=\frac{a^{q^e}+b^{q^e}-(a+b)^{q^e}}{\varpi^e}$.
\end{lemma} 
\begin{proof}[Proof of Lemma \ref{mainlem}] Our method is to find a concrete required element $g$ in all possible cases. We remark firstly that by (\ref{formulaT1}), $T_+R_1(\sigma)$ is spanned (over $k$) by elements
	$$\sum_{\lambda \in \F_q}\left[\mtwo{\varpi^2}{[\mu]+\varpi [\lambda]}{}{1}, \lambda^{\vec{i}}x^{\vec{r}}\right]$$
	where $\vec{0}\leq \vec{i}\leq \vec{r}$ and $\mu\in \F_q$. Moreover 
	$\sum_{\lambda \in \F_q}\left[\mtwo{\varpi^2}{[\mu]+\varpi [\lambda]}{}{1}, \lambda^{\vec{i}}x^{\vec{r}}\right]$ lies in $T_+R_1'(\sigma)$ if $\vec{i}<\vec{r}$ by (\ref{formulaT1}) and (\ref{formulaT2}), here $\vec{i}<\vec{r}$ means $\vec{i}\leq \vec{r}$ and $\vec{i}\neq \vec{r}$. Since $T_{\pm}$ are $U$-equivariant, $R_1'(\sigma)$ and $T_+ R_1'(\sigma)$ are stable under the action of $U$. Moreover, $\alpha^3 U \alpha^{-3}=\Uni{\varpi^3\cO}$ acts trivially on $R_2(\sigma)$.\par
	1) Assume $F$ is ramified over $\Q_p$ with $e\geq 2$. We have 
	$$[a]+[b]\equiv [a+b]\mod \varpi^{2},$$
	by Lemma \ref{lemmawitt}.\par
	If $\text{dim}_k (\sigma)> 1$, there exists $j_0$ such that $r_{j_0}\geq 1$. Let $\vec{i'}=(i_0',\cdots, i_{f-1}' )\in \Z^{f}$ where $i_{j}'=0$ if $j\neq j_0$ and $i_{j_0}'=1$. Then $\vec{i'}\leq \vec{r}$. We take
	$$g=\sum_{\mu,\lambda \in \F_q} \left[\mtwo{\varpi^2}{[\mu]+\varpi [\lambda]}{}{1}, x^{\vec{r}-\vec{i'}}y^{\vec{i'}}\right].$$
	Then $g\notin T_+R_1(\sigma)$.
	For $a\in \F_q$, we calculate that 
	$$\begin{aligned}
		\Uni{\varpi[a]}g-g&=\sum_{\mu,\lambda \in \F_q}\left [\mtwo{\varpi^2}{[\mu]+\varpi [a]+\varpi [\lambda]}{}{1}, x^{\vec{r}-\vec{i'}}y^{\vec{i'}}\right]-g\\
		&=\sum_{\mu,\lambda \in \F_q} \left[\mtwo{\varpi^2}{[\mu]+\varpi [a+\lambda]}{}{1}, \Uni{\frac{[a]+[\lambda]-[a+\lambda]}{\varpi}}x^{\vec{r}-\vec{i'}}y^{\vec{i'}}\right]-g\\
		&=\sum_{\mu,\lambda \in \F_q} \left[\mtwo{\varpi^2}{[\mu]+\varpi [a+\lambda]}{}{1}, \Uni{\varpi \cdot\frac{[a]+[\lambda]-[a+\lambda]}{\varpi^2}}x^{\vec{r}-\vec{i'}}y^{\vec{i'}}\right]-g\\
		&=\sum_{\mu,\lambda \in \F_q} \left[\mtwo{\varpi^2}{[\mu]+\varpi [a+\lambda]}{}{1}, x^{\vec{r}-\vec{i'}}y^{\vec{i'}}\right]-g\\
		&=0
	\end{aligned} 
	$$
	For all $a,b,\mu\in \F_q$, let $t_{a,b,\mu}$ be the image of $[b]+\frac{[a]+[\mu]-[a+\mu]}{\varpi^2}$ in $\F_q$, then
	$$
	\begin{aligned}
		\Uni{[a]+\varpi^2[b]}g-g
		&=\sum_{\mu,\lambda \in \F_q} \left[\mtwo{\varpi^2}{[a]+[\mu]+\varpi^2 [b]+\varpi [\lambda]}{}{1}, x^{\vec{r}-\vec{i'}}y^{\vec{i'}}\right]-g\\
		&=\sum_{\mu,\lambda \in \F_q} \left[\mtwo{\varpi^2}{[\mu+a]+\varpi [\lambda]}{}{1}, \Uni{[b]+\frac{[a]+[\mu]-[a+\mu]}{\varpi^2}}x^{\vec{r}-\vec{i'}}y^{\vec{i'}}-x^{\vec{r}-\vec{i'}}y^{\vec{i'}}\right]\\
		&=\sum_{\mu,\lambda \in \F_q} \left[\mtwo{\varpi^2}{[\mu+a]+\varpi [\lambda]}{}{1}, t_{a,b,\mu}^{p^{j_0}}x^{\vec{r}}\right]\\
		&=T_+\left(\sum_{\mu \in \F_q} \left[\mtwo{\varpi}{[\mu]}{}{1}, t_{a,b,\mu-a}^{p^{j_0}}x^{\vec{r}}\right]\right)\in T_+R_1'.
	\end{aligned}
	$$
	Since $\Uni{\varpi[a]},\Uni{[a]+\varpi^2[b]}, a,b\in\F_q$ generate $U/\alpha^3U\alpha^{-3}$, we see that $g\in (R_2(\sigma)/T_+R_1'(\sigma))^{U}$ and $g\notin T_{+}R_1(\sigma)$.\par 
	If $\text{dim}_k(\sigma)=1$, $\vec{r}=\vec{0}$. We take $g=\sum_{\mu,\lambda \in \F_q}\left[\mtwo{\varpi^2}{[\mu]+\varpi[\lambda]}{}{1},\lambda\right]$. Then $g\notin T_+R_1(\sigma)$ as $T_+R_1(\sigma)$ is spanned by $\sum_{\lambda\in\F_q}\left[\mtwo{\varpi^2}{[\mu]+\varpi[\lambda]}{}{1},1\right]$ by (\ref{formulaT1}). Then for any $a,b,c\in\F_q$,
	$$
	\begin{aligned}
	&\quad\Uni{[a]+\varpi [b]+\varpi^2[c]}g-g\\
	&=\sum_{\mu,\lambda \in \F_q} \left[\mtwo{\varpi^2}{[a+\mu]+\varpi [\lambda+b]}{}{1}, \Uni{[c]+\frac{[a]+[\mu]-[a+\mu]}{\varpi^2}+\frac{[b]+[\lambda]-[b+\lambda]}{\varpi}}\lambda\right]-g\\
	&=\sum_{\mu,\lambda \in \F_q} \left[\mtwo{\varpi^2}{[a+\mu]+\varpi [\lambda+b]}{}{1}, \lambda-(\lambda+b)\right]\\
	&=\sum_{\mu,\lambda \in \F_q} \left[\mtwo{\varpi^2}{[a+\mu]+\varpi [\lambda+b]}{}{1}, -b\right]\\
	&= T_+\left(\sum_{\mu \in \F_q} \left[\mtwo{\varpi}{[\mu]}{}{1}, -b\right]\right)\in T_+R_1'(\sigma)
	\end{aligned}
	$$ 
	since $T_-\left(\sum_{\mu \in \F_q} \left[\mtwo{\varpi}{[\mu]}{}{1}, -b\right]\right)=\nu\left[\mtwo{1}{}{}{1}, \sum_{\mu\in\F_q}-b\right]=0$ by (\ref{formulaT2}).

	2) Assume $F$ is unramified. Then $f> 1$, $\varpi =p$. By the theory of Witt vectors, there exist polynomials $P_1,P_2\in\Z[x,y]$ such that for any $a,b\in\F_q$, $[a]+[b]\equiv [a+b]+p[P_1(a,b)]+p^{2}[P_2(a,b)]\mod p^{3}$. Since $P_1(a,b)=F(a^{1/p},b^{1/p})=F(a^{p^{f-1}},b^{p^{f-1}})$ where $F(x,y)=\frac{x^p+y^p-(x+y)^p}{p}$, we can assume $P_1$ is a polynomial of degree no more than $p^{f-1}(p-1)$ in each variable (or see Lemma \ref{lemmawitt}).\par 
	If there exists $j_0\in \{0,\cdots, f-1\}$ such that $r_{j_0}+1\leq p-1$ (i.e. $\vec{r}\neq (p-1,\cdots,p-1)$), we take 
	$$g=\sum_{\mu,\lambda \in \F_q}\left[\mtwo{p^2}{[\mu]+p[\lambda]}{}{1},\lambda^{p^{j_0}(r_{j_0}+1)}x^{\vec{r}}\right]. $$
	We claim that $g\notin T_+R_1(\sigma)$. Otherwise, for each $\mu\in\F_q$, there exist $u_{\vec{i}}\in k$ for $\vec{0}\leq \vec{i}\leq \vec{r}$ such that $$\sum_{\lambda \in \F_q}\left[\mtwo{p^2}{[\mu]+p[\lambda]}{}{1},\lambda^{p^{j_0}(r_{j_0}+1)}x^{\vec{r}}\right]=\sum_{\lambda \in \F_q}\left[\mtwo{p^2}{[\mu]+p[\lambda]}{}{1}, (\sum_{0\leq\vec{i}\leq \vec{r}}u_{\vec{i}}{(-\lambda)}^{\vec{i}})x^{\vec{r}}\right].$$
	Then $\lambda^{p^{j_0}(r_{j_0}+1)}=\sum_{\vec{i}\leq\vec{r}}u_{\vec{i}}(-1)^{\vec{i}}\lambda^{\vec{i}}$ for every $\lambda\in\F_q$. This is impossible since the polynomial $X^{p^{j_0}(r_{j_0}+1)}-\sum_{0\leq\vec{i}\leq \vec{r}}u_{\vec{i}}(-1)^{\vec{i}}X^{\sum_{0\leq j\leq f-1}p^{j}i_j}\in k[X]$ is not zero and has degree no more than $q-2$ (by $f>1$ and $\vec{r}\neq (p-1,\cdots,p-1)$).
	For any $a,b,c\in \F_q$, we calculate that (using $x^{\vec{r}}\in \sigma^U$ and $[a]+[b]\equiv[a+b]+p[P_1(a,b)]\mod p^2$)
	\begin{align}\nonumber
		&\quad\Uni{[a]+p[b]+p^2[c]}g-g\\\nonumber
		&=\sum_{\mu,\lambda \in \F_q} \left[\mtwo{p^2}{[a]+[\mu]+p[\lambda]+p[b]+p^2[c]}{}{1}, \lambda^{p^{j_0}(r_{j_0}+1)}x^{\vec{r}}\right]-g\\\nonumber
		&=\sum_{\mu,\lambda \in \F_q} \left[\mtwo{p^2}{[a+\mu]+p[\lambda+b+P_1(a,\mu)]}{}{1}, \lambda^{p^{j_0}(r_{j_0}+1)}x^{\vec{r}}\right]-g\\\label{lastterm}
		&=\sum_{\mu,\lambda \in \F_q} \left[\mtwo{p^2}{[\mu]+p[\lambda]}{}{1}, ((\lambda-b-P_1(a,\mu-a))^{p^{j_0}(r_{j_0}+1)}-\lambda^{p^{j_0}(r_{j_0}+1)})x^{\vec{r}}\right],
	\end{align}
	Write $(\lambda-b-P_1(a,\mu-a))^{p^{j_0}(r_{j_0}+1)}-\lambda^{p^{j_0}(r_{j_0}+1)}=\sum_{0 \leq i\leq r_{j_0}}g_{i}(\mu) (-\lambda) ^{p^{j_0}i}$, where $g_{i}(\mu)$ are polynomials in $\mu$ (depending also on $a,b$).\par
	First assume $p^{j_0}r_{j_0}\neq r=\sum_{j=0}^{f-1} r_jp^j$. For any $0\leq i\leq r_{j_0}$, let $\vec{i}_{j_0}=(i_1,\cdots,i_{f-1})\in \Z^{f-1}$ such that $i_j=0$ if $j\neq j_0$ and $i_{j_0}=i$. Then $\vec{i}_{j_0}<\vec{r}$ for any $i\leq r_{j_0}$. Hence the last term in (\ref{lastterm})
	\begin{align*}
	  \sum_{\mu,\lambda \in \F_q}& \left[\mtwo{p^2}{[\mu]+p[\lambda]}{}{1}, (\sum_{0 \leq i\leq r_{j_0}}g_{i}(\mu) (-\lambda) ^{p^{j_0}i})x^{\vec{r}}\right]\\=&\sum_{\mu \in \F_q}T_+\left(\left[\mtwo{p}{[\mu]}{}{1}, \sum_{0 \leq i\leq r_{j_0}}g_{i}(\mu) x^{\vec{r}-\vec{i}_{j_0}}y^{\vec{i}_{j_0}}\right]\right)
	\end{align*}
	lies in $T_+R_1'$ and we have found a required $g$.\par
	Otherwise $r=p^{j'}r_{j'}$ for some $j'$. If $\vec{r}\neq 0$, we can choose in the beginning $j_0\neq j'$ with $r_{j_0}=0$ since $f\geq 2$ and $r_{j_0}+1=1\leq p-1$. Then $0=p^{j_0}r_{j_0}\neq r$, we return to the previous case and we can find a required $g$. If $\vec{r}=0$, we can let $j_0=0$, then $r_{j_0}=0$. Then the last term in (\ref{lastterm}) is 
	$\sum_{\mu,\lambda \in \F_q} \left[\mtwo{p^2}{[\mu]+p[\lambda]}{}{1},g_0(\mu)\right]=T_+\left(\sum_{\mu\in \F_q} \left[\mtwo{p}{[\mu]}{}{1},g_0(\mu)\right]\right)$.  We have $$T_-\left(\sum_{\mu\in \F_q} \left[\mtwo{p}{[\mu]}{}{1},g_0(\mu)\right]\right)=\nu\left[\mtwo{1}{}{}{1},\sum_{\mu\in \F_q} g_0(\mu)\right]=0$$ by Lemma \ref{lemmasumfq} and $g_0(\mu)$ is a polynomial of $\mu$ of degree $(p-1)p^{f-1}<q-1$. Hence $ug-g\in T_+R_1'(\sigma)$ for any $u\in U$. We have found a required $g$.
	\par
	(3) Now we remain the case when $F$ is unramified over $\Qp$, $f\geq 2$ and $\vec{r}=(p-1,\cdots, p-1)$. Let $\vec{i'}=(i_0',\cdots, i_{f-1}' )$ where $i_{j}'=0$ if $j\neq 0$ and $i_{0}'=1$. Take 
	$$g=\sum_{\mu,\lambda \in \F_q} \left[\mtwo{p^2}{[\mu]+p[\lambda]}{}{1}, x^{\vec{r}-\vec{i'}}y^{\vec{i'}}\right].$$
	Then $g \notin T_+R_1$ as $\vec{i'}\neq \vec{0}$. For any $a,b\in\F_q$, we calculate that (using $\Uni{a}x^{\vec{r}-\vec{i'}}y^{\vec{i'}}=ax^{\vec{r}}+x^{\vec{r}-\vec{i'}}y^{\vec{i'}}$) 
	$$\begin{aligned}
		&\quad\Uni{p[a]+p^2[b]}g-g\\
		&=\sum_{\mu,\lambda \in \F_q} \left[\mtwo{p^2}{[\mu]+p[a]+p [\lambda]+p^2[b]}{}{1}, x^{\vec{r}-\vec{i'}}y^{\vec{i'}}\right]-g\\
		&=\sum_{\mu,\lambda \in \F_q} \left[\mtwo{p^2}{[\mu]+p[a+\lambda]+p^2[P_1(a,\lambda)]+p^2[b]+p^3\frac{[a]+[\lambda]-[a+\lambda]-p[P_1(a,\lambda)]}{p^2}}{}{1}, x^{\vec{r}-\vec{i'}}y^{\vec{i'}}\right]-g\\
		&=\sum_{\mu,\lambda \in \F_q} \left[\mtwo{p^2}{[\mu]+p [a+\lambda]}{}{1}, \Uni{[P_1(a,\lambda)]+[b]}x^{\vec{r}-\vec{i'}}y^{\vec{i'}}\right]-g\\
		&=\sum_{\mu,\lambda \in \F_q} \left[\mtwo{p^2}{[\mu]+p [\lambda]}{}{1}, (P_1(a,\lambda-a)+b)x^{\vec{r}}\right].
	\end{aligned} 
	$$
	$P_1(a,\lambda-a)+b$ is a polynomial of $\lambda$ with degree no more than $p^{f-1}(p-1)<q-1$, the last term lies in $T_+R_1'$ by the remark at the beginning ($\sum_{\lambda \in \F_q}\left[\mtwo{p^2}{[\mu]+\varpi [\lambda]}{}{1}, \lambda^{\vec{i}}x^{\vec{r}}\right]$ lies in $T_+R_1'(\sigma)$ if $\vec{i}<\vec{r}$). \par
	For any $a\in\F_q$,
	$$
	\begin{aligned}
		&\quad\Uni{[a]}g-g\\
		&=\sum_{\mu,\lambda \in \F_q} \left[\mtwo{p^2}{[a]+[\mu]+p[\lambda]}{}{1}, x^{\vec{r}-\vec{i'}}y^{\vec{i'}}\right]-g\\
		&=\sum_{\mu,\lambda \in \F_q} \left[\mtwo{p^2}{[\mu+a]+p [\lambda+{P_1(a,\mu)}]}{}{1}, \Uni{[P_2(a,\mu)]+[P_1(\lambda,P_1(a,\mu))]}x^{\vec{r}-\vec{i'}}y^{\vec{i'}}-x^{\vec{r}-\vec{i'}}y^{\vec{i'}}\right]\\
		&=\sum_{\mu,\lambda \in \F_q} \left[\mtwo{p^2}{[\mu+a]+p [\lambda+{P_1(a,\mu)}]}{}{1}, (P_2(a,\mu)+P_1(\lambda,P_1(a,\mu)))x^{\vec{r}}\right]\\
		&=\sum_{\mu,\lambda \in \F_q} \left[\mtwo{p^2}{[\mu]+p[\lambda]}{ }{1},(P_2(a,\mu-a)+P_1(\lambda-P_1(a,\mu-a),P_1(a,\mu-a)))x^{\vec{r}}\right].
	\end{aligned}
	$$
	$(P_2(a,\mu-a)+P_1(\lambda-P_1(a,\mu-a),P_1(a,\mu-a)))$ is a polynomial of $\lambda$ of degree no more than $p^{f-1}(p-1)<q-1$. By the remark at the beginning, the last term lies in $T_+R_1'(\sigma)$. \par
	Since $\Uni{[a]},\Uni{p[b]+p^2[c]},a,b,c\in \F_q$ generate $U/\alpha^3U\alpha^{-3}$, $g\in (R_2(\sigma)/T_+R_1'(\sigma))^U$. Thus we have found a required $g$.
	\end{proof}
	\begin{remark}
	    Those $g$ in Lemma \ref{mainlem} have been found for many cases in \cite{Breuil2003gl2}, \cite{schein2011irreducibility}, \cite{morra2012some} and \cite{hendel2019universal}.
	\end{remark}
	\begin{corollary}\label{maincor}
		For any smooth irreducible representation $\sigma$ of $KZ$, the universal supersingular representation of $G$ $\ind_{KZ}^G\sigma/T(\ind_{KZ}^G\sigma)$ is not admissible if $F\neq \Qp$.
	\end{corollary}
	\begin{proof}
		Same as Corollary 2.21 in \cite{schraen2015presentation}, using Proposition 4.5 in \cite{emerton2008class}. 
	\end{proof}

\bibliographystyle{plain}	

\bibliography{note}	
\end{document}